\documentclass{amsart}

\usepackage{latexsym,amsmath,amssymb,amsfonts,amscd,graphics,appendix,amsxtra}
\usepackage[mathscr]{eucal}

\newtheorem{theorem}{Theorem}[section]
\newtheorem{proposition}[theorem]{Proposition}
\newtheorem{lemma}[theorem]{Lemma}
\newtheorem{claim}[theorem]{Claim}
\newtheorem{corollary}[theorem]{Corollary}

\newtheorem{D}[theorem]{Definition}
\newenvironment{definition}{\begin{D} \rm }{\end{D}}
\newtheorem{R}[theorem]{Remark}
\newenvironment{remark}{\begin{R}\rm }{\end{R}}
\newtheorem{E}[theorem]{Example}
\newenvironment{example}{\begin{E}\rm }{\end{E}}

\def\Zee{\mathbb{Z}}
\def\Q{\mathbb{Q}}
\def\Ar{\mathbb{R}}
\def\Cee{\mathbb{C}}
\def\Pee{\mathbb{P}}

\def\Ker{\operatorname{Ker}}

\def\Aut{\operatorname{Aut}}
\def\Pic{\operatorname{Pic}}

\def\scrO{\mathcal{O}}
\def\spcheck{^{\vee}}

\title[the ample cone of a rational surface]{On the ample cone of a rational surface \\
with an anticanonical cycle} 
\date{\today}
\begin{document}
\author[R. Friedman]{Robert Friedman}
\address{Department of Mathematics, Columbia University,  New York, NY 10027}
\email{rf@math.columbia.edu}
\keywords{rational surface, anticanonical cycle, exceptional curve, ample cone}
\subjclass{14J26}

\begin{abstract}
Let $Y$ be a smooth rational surface and let $D$ be a cycle of rational curves on $Y$ which is an anticanonical divisor, i.e.\  an element of $|-K_Y|$. Looijenga studied the geometry of such surfaces $Y$ in case $D$ has at most five components and identified a geometrically significant subset $R$ of the divisor classes of square $-2$ orthogonal to the components of $D$. Motivated by recent work of Gross, Hacking, and Keel on the global Torelli theorem for pairs $(Y,D)$, we attempt to generalize some of Looijenga's results in case $D$ has more than five components. In particular, given an integral isometry $f$ of $H^2(Y)$ which preserves the classes of the components of $D$, we investigate the relationship between the condition that $f$ preserves the ``generic" ample cone of $Y$ and the condition that $f$ preserves the set $R$.
\end{abstract}

\maketitle

\section*{Introduction}

The ample cone of a del Pezzo surface $Y$ (or rather the associated dual polyhedron) was studied classically by, among others, Gosset, Schoute, Kantor, Coble, Todd, Coxeter, and Du Val. For a brief historical discussion, one can consult the remarks in \S11.x of \cite{Coxeter}. From this point of view, the lines on $Y$ are the main object of geometric interest, as they are the walls of the ample cone or the vertices of  the  dual polyhedron. The corresponding root system (in case $K_Y^2 \leq 6$) only manifests itself geometrically by allowing del Pezzo surfaces with rational double points, or equivalently smooth surfaces $Y$ with $-K_Y$ nef and big but not ample. This is explicitly worked out in Part II of Du Val's series of papers \cite{duV}. On the other hand, the root system, or rather its Weyl group, appears for a smooth del Pezzo surface as a group of symmetries of the ample cone, a fact which (in a somewhat different guise) was already known to Cartan. Perhaps the culmination of the classical side of the story is Du Val's 1937 paper \cite{duV2}, where he also systematically considers the blowup of $\Pee^2$ at $n\geq 9$ points. In modern times, Manin explained the appearance of the Weyl group by noting that the orthogonal complement to $K_Y$ in $H^2(Y;\Zee)$ is a root lattice $\Lambda$. Moreover, given any root of $\Lambda$, in other words an element $\beta$ of square $-2$, there exists a deformation of $Y$ for which $\beta = \pm[C]$, where $C$ is a smooth rational curve of self-intersection $-2$. For modern expositions of the theory, see for example Manin's book \cite{Manin} or Demazure's account in \cite{777}. 

In general, it seems hard to study an arbitrary rational surface $Y$ without imposing some extra conditions.  One very natural condition is that $-K_Y$ is effective, i.e.\ that $-K_Y = D$ for an effective divisor $D$.  In case the intersection matrix of $D$ is negative definite, such pairs $(Y,D)$ arise naturally in the study of minimally elliptic singularities: the case where $D$ is a smooth elliptic curve corresponds to the case of simple elliptic singularities, the case where $D$ is a nodal curve or a cycle of smooth rational curves meeting transversally corresponds to the case of cusp singularities, and the case where $D$ is reduced but has  one component with a cusp, two components with a  tacnode  or   three components meeting at a point, corresponds to triangle singularities. From this point of view, the case where $D$ is a cycle of rational curves is the most plentiful.  The systematic study of such surfaces in case the intersection matrix of $D$ is negative definite dates back to Looijenga's seminal paper \cite{Looij}.  However, for various technical reasons, most of the results of that paper are proved under the assumption that the number of components in the cycle is at most $5$. Some of the main points of \cite{Looij} are as follows: Denote by $R$ the set of elements in $H^2(Y; \Zee)$ of square $-2$ which are orthogonal to the components of $D$ and which are of the form $\pm [C]$, where $C$ is a smooth rational curve disjoint from $D$, for some deformation of the pair $(Y,D)$. In terms  of deformations of singularities, the set $R$ is related to the possible rational double point singularities which can arise as deformations of the   dual cusp to the cusp singularity corresponding to $D$. Looijenga noted that, in general,  there exist elements in $H^2(Y; \Zee)$ of square $-2$ which are orthogonal to the components of $D$ but which do not lie in $R$.  Moreover, reflections in elements of the set $R$ give symmetries of the ``generic" ample cone (which is the same as the ample cone in case there are no smooth rational curves on $Y$ disjoint from $D$). Finally, still under the assumption of at most 5 components, any isometry of $H^2(Y; \Zee)$ which preserves the  positive cone, the classes $[D_i]$ and the set $R$, preserves the generic ample cone. 

This paper, which is an attempt to see how much of \cite{Looij} can be generalized to the case of arbitrarily many components, is motivated by a question raised by the recent work of Gross, Hacking and Keel \cite{GHK} on, among matters, the global Torelli theorem for pairs $(Y,D)$ where $D$ is an anticanonical cycle on the rational surface $Y$. In order to formulate this theorem in a fairly general way, one would like  to characterize the isometries $f$ of $H^2(Y, \Zee)$, preserving the  positive cone and fixing the classes $[D_i]$, which preserve the ample cone of $Y$. It is natural to ask if, at least in the generic case, the condition  that  $f(R)=R$ is sufficient. In this paper, we give various criteria on $R$ which insure that, if an isometry $f$ of $H^2(Y; \Zee)$   preserves the positive cone, the classes $[D_i]$ and the set $R$, then $f$ preserves the generic ample cone. Typically, one needs a hypothesis which says that $R$ is large. For example, one such hypothesis is that there is a subset of $R$ which spans a negative definite codimension one subspace of the orthogonal complement to the components of $D$.  
In theory, at least under various extra hypotheses, such a result gives a necessary and sufficient condition for an isometry to preserve the generic ample cone. In practice, however, the determination of the set $R$ in general is a difficult problem, which seems close in its complexity to the problem of describing the generic ample cone of $Y$. Finally, we show that some assumptions on $(Y,D)$ are necessary, by giving examples where $R=\emptyset$, so that the condition that an isometry $f$ preserves $R$ is automatic, and of isometries $f$ such that $f$ preserves the   positive cone, the classes $[D_i]$ and (vacuously) the set $R$, but $f$ does not preserve the generic ample cone. We do not yet have a good  understanding  of the  relationship between preserving the ample cone and preserving the set $R$.

An outline of this paper is as follows. The preliminary Section 1  reviews standard methods for constructing nef classes on algebraic surfaces and applies this to the study of when the normal surface obtained by contracting a negative definite anticanonical cycle on a rational surface is projective. In Section 2, we analyze the ample cone and generic ample cone of a pair $(Y,D)$ and show that  the set $R$ defined by Looijenga is exactly the set of elements $\beta$ in $H^2(Y; \Zee)$ of square $-2$ which are orthogonal to the components of $D$ such that reflection about $\beta$ preserves the generic ample cone. Much of the material of \S2 overlaps with results in \cite{GHK}, proved there by somewhat different methods.  Section 3 is devoted to giving various sufficient conditions for an isometry $f$ of $H^2(Y; \Zee)$ to preserve the generic ample cone, including the one described above. Section 4 gives  examples of pairs $(Y,D)$ satisfying the sufficient conditions of \S3 where the number of components of $D$ and the multiplicity $-D^2$ are arbitrarily large, as well as examples showing that some hypotheses on $(Y,D)$ are necessary.   

\medskip
\noindent\textbf{Acknowledgements.} It is a pleasure to thank Mark Gross, Paul Hacking and Sean Keel for access  to their manuscript \cite{GHK} and for extremely stimulating correspondence and conversations about these and other matters, and Radu Laza  for many helpful discussions.

\medskip
\noindent\textbf{Notation and conventions.} We work over $\Cee$. If $X$ is a smooth projective surface with $h^1(\scrO_X) = h^2(\scrO_X) = 0$ and $\alpha \in H^2(X; \Zee)$, we denote by $L_\alpha$ the corresponding holomorphic line bundle, i.e.\ $c_1(L_\alpha) = \alpha$. Given a curve $C$ or divisor class $G$ on $X$, we denote by $[C]$ or $[G]$ the corresponding element of $H^2(X; \Zee)$. Intersection pairing on curves or divisors, or on elements in the second cohomology of a smooth surface (viewed as a canonically oriented $4$-manifold) is denoted by multiplication.

\section{Preliminaries}

Throughout this paper,  $Y$ denotes a smooth rational surface with  $-K_Y = D= \sum_{i=1}^rD_i$ a (reduced) cycle of rational curves, i.e.\ each $D_i$ is a smooth rational curve and $D_i$ meets  $D_{i\pm 1}$ transversally, where $i$ is taken mod $r$, except for $r=1$, in which case $D_1=D$ is an irreducible nodal curve. We note, however, that many of the results in this paper can be generalized to the case where $D\in |-K_Y|$ is not assumed to be a cycle. The integer $r=r(D)$ is called the \textsl{length} of $D$. An \textsl{orientation} of $D$ is an orientation of the dual graph (with appropriate modifications in  case $r=1$). We shall abbreviate the data of the surface $Y$ and the oriented cycle $D$ by $(Y,D)$ and refer to it as a \textsl{anticanonical pair}. If the intersection matrix $(D_i\cdot D_j)$ is negative definite, we say that $(Y,D)$ is a \textsl{negative definite anticanonical pair}. 

\begin{definition}\label{defcurves} An irreducible curve $E$ on $Y$ is  an \textsl{exceptional curve} if $E\cong \Pee^1$, $E^2 = -1$, and $E \neq D_i$ for any $i$. An irreducible curve $C$ on $Y$ is  a \textsl{$-2$-curve}  if $C\cong \Pee^1$,  $C^2 = -2$, and $C \neq D_i$ for any $i$. Let   $\Delta_Y$ be the set of all $-2$-curves on $Y$, and let $\mathsf{W}({\Delta_Y})$ be the group of integral isometries of $H^2(Y; \Ar)$ generated by the reflections in the classes in the set  $\Delta_Y$.
\end{definition}

\begin{definition} Let $\Lambda = \Lambda(Y,D) \subseteq H^2(Y; \Zee)$ be the orthogonal complement of the lattice spanned by the classes $[D_i]$. Fixing the identification $\Pic^0D \cong \mathbb{G}_m$ defined by the orientation of the cycle $D$, we define the \textsl{period map} $\varphi_Y \colon \Lambda \to \mathbb{G}_m$ via: if $\alpha \in \Lambda$ and $L_\alpha$ is the corresponding line bundle, then $\varphi_Y(\alpha) \in \mathbb{G}_m$ is the image of the line bundle of multi-degree $0$ on $D$ defined by $L_\alpha|D$. Clearly $\varphi_Y$ is a homomorphism.
\end{definition}

By \cite{Looij}, \cite{FriedmanScattone}, \cite{Fried2}, we have:

\begin{theorem}\label{surjper} The period map is surjective. More precisely, given $Y$ as above and given an arbitrary homomorphism $\varphi \colon \Lambda \to\mathbb{G}_m$, there exists a deformation of the pair $(Y,D)$ over a smooth connected base, which we can take to be a product of $\mathbb{G}_m$'s, such that the monodromy of the family is trivial and there exists a fiber of the deformation, say $(Y', D')$ such that, under the induced identification of $\Lambda(Y',D')$ with $\Lambda$, $\varphi_{Y'} = \varphi$. \qed
\end{theorem}

For future reference, we recall some standard facts about negative definite curves on a surface:

\begin{lemma}\label{negdef} Let $X$ be a smooth projective surface and let $G_1, \dots , G_n$ be irreducible curves on $X$ such that the intersection matrix $(G_i\cdot G_j)$ is negative definite. Let $F$ be an effective divisor on $X$, not necessarily reduced or irreducible, and such that, for all $i$,  $G_i$ is not a  component of $F$. 
\begin{enumerate}
\item[\rm{(i)}] Given $r_i \in \Ar$, if $(F + \sum_ir_iG_i) \cdot G_j = 0$ for all $j$, then $r_i \geq  0$ for all $i$, and, for every subset $I$ of $\{1, \dots, n\}$, if $\bigcup_{i\in I}G_i$ is a connected curve such that $F\cdot G_j \neq 0$ for some $j\in I$, then $r_i > 0$ for $i\in I$.
\item[\rm{(ii)}] Given $s_i, t_i \in \Ar$, if $[F] + \sum_is_i[G_i] = \sum_it_i[G_i]$, then $F=0$ and $s_i = t_i$ for all $i$. \qed
\end{enumerate}
\end{lemma}

The following general result is also well-known:

\begin{proposition}\label{constructnef} Let $X$ be a smooth projective surface and let $G_1, \dots , G_n$ be irreducible curves on $X$ such that the intersection matrix $(G_i\cdot G_j)$ is negative definite. (We do not, however, assume that $\bigcup_iG_i$ is connected.) 
Then there exists a nef and big divisor $H$ on $X$ such that $H\cdot G_j = 0$ for all $j$ and, if $C$ is an irreducible curve such that $C \neq G_j$ for any $j$, then  $H\cdot C >0$. In fact, the set of nef and big $\Ar$-divisors which are orthogonal to $\{G_1, \dots, G_n\}$ is a nonempty open subset of $\{G_1, \dots, G_n\}^\perp \otimes \Ar$.
\end{proposition}
\begin{proof} Fix an ample divisor $H_0$ on $X$. Since $(G_i\cdot G_j)$ is negative definite, there exist $r_i\in \Q$ such that $(\sum_ir_iG_i) \cdot G_j = -(H_0\cdot G_j)$ for every $j$, and hence $(H_0 + \sum_ir_iG_i) \cdot G_j = 0$. By Lemma~\ref{negdef}, $r_i > 0$ for every $i$. There exists an $N > 0$ such that   $Nr_i \in \Zee$ for all $i$. Then $H = N(H_0 + \sum_ir_iG_i)$ is an effective  divisor satisfying $H\cdot G_j = 0$ for all $j$. If $C$ is an irreducible curve such that $C \neq G_j$ for any $j$, then $H_0 \cdot C > 0$ and $G_i \cdot C \geq 0$ for all $i$, hence $H\cdot C >0$. In particular $H$ is nef. Finally $H$ is big since $H^2 = NH\cdot(H_0 + \sum_ir_iG_i) = N(H\cdot H_0) > 0$, as $H_0$ is ample.

To see the final statement, we   apply the above argument to an ample $\Ar$-divisor $x$ (i.e.\ an element in the interior of the ample cone) to see that $x + \sum_ir_iG_i$ is a nef  and big $\Ar$-divisor orthogonal to $\{G_1, \dots, G_n\}$. Since $x + \sum_ir_iG_i$ is simply the orthogonal projection $p$ of $x$ onto $\{G_1, \dots, G_n\}^\perp \otimes \Ar$, and $p\colon H^2(X; \Ar) \to  \{G_1, \dots, G_n\}^\perp \otimes \Ar$ is an open map, the image of the interior of the ample cone of $X$ is then a nonempty open subset of $\{G_1, \dots, G_n\}^\perp \otimes \Ar$ consisting of nef and big $\Ar$-divisors  orthogonal to $\{G_1, \dots, G_n\}$.
\end{proof}

Applying the above construction to $X=Y$ and $D_1, \dots, D_r$, we can find a nef and big divisor $H$ such that $H\cdot D_j = 0$ for all $j$ and such that, if $C$ is an irreducible curve such that $C \neq D_j$ for any $j$, then  $H\cdot C >0$.

\begin{proposition}
 Let $(Y,D)$ be a negative definite anticanonical pair and let $H$ be  a nef and big divisor such that $H\cdot D_j = 0$ for all $j$ and such that, if $C$ is an irreducible curve such that $C \neq D_j$ for any $j$, then  $H\cdot C >0$. Suppose in addition that $\scrO_Y(H)|D = \scrO_D$, i.e.\ that $\varphi_Y([H]) =1$. Then the $D_i$ are not fixed components of $|H|$. Hence, if $\overline{Y}$ denotes the normal complex surface obtained by contracting the $D_i$, then $H$ induces an  ample divisor $\overline{H}$ on $\overline{Y}$ and $|3\overline{H}|$ defines an embedding of $\overline{Y}$ in $\Pee^N$ for some $N$.  
\end{proposition}
\begin{proof} Consider the exact sequence
$$0 \to \scrO_Y(H-D) \to \scrO_Y(H) \to \scrO_D \to 0.$$
Looking at the long exact cohomology sequence, as 
$$H^1(Y; \scrO_Y(H-D)) = H^1(Y;\scrO_Y(H) \otimes K_Y)$$ is Serre dual to $H^1(Y; \scrO_Y(-H)) = 0$, by Mumford vanishing, there exists a section of $\scrO_Y(H)$ which is nowhere vanishing on $D$, proving the first statement. The second follows from Nakai-Moishezon and the third from general results on linear series on anticanonical pairs \cite{Fried1}.
\end{proof}

\begin{remark} By the surjectivity of the period map \ref{surjper}, for any $(Y,D)$ a negative definite anticanonical pair and $H$ a nef and big divisor on $Y$ such that $H\cdot D_j = 0$ for all $j$ and $H\cdot C > 0$ for all curves $C\neq D_i$, there exists a deformation of the pair $(Y,D)$ such that the divisor corresponding to $H$ has  trivial restriction to $D$. More generally, one can consider deformations such that $\varphi_Y([H])$ is a torsion point of $\mathbb{G}_m$. In this case, if $\overline{Y}$ is the normal surface obtained by contracting $D$, then $\overline{Y}$ is projective. Note that this implies that the set of pairs $(Y,D)$ such that  $\overline{Y}$ is projective is Zariski dense in the moduli space. However, as the set of torsion points is not dense in $\mathbb{G}_m$ in the classical topology, the set of projective surfaces $\overline{Y}$ will not be dense in the classical topology.
\end{remark}

\section{Roots and nodal classes}

\begin{definition}
Let $\mathcal{C}=\mathcal{C}(Y)$ be the positive cone of $Y$, i.e.\ 
$$\mathcal{C} = \{x\in H^2(Y; \Ar): x^2 >0\}.$$
 Then $\mathcal{C}$ has two components, and exactly one of them, say $\mathcal{C}^+=\mathcal{C}^+(Y)$, contains the classes of ample divisors. We also define
$$\mathcal{C}^+_D =  \mathcal{C}^+_D(Y) = \{x\in \mathcal{C}^+: x \cdot [D_i] \geq 0 \text{ for all $i$ }\}.$$
Let $\overline{\mathcal{A}}(Y)\subseteq  \mathcal{C}^+ \subseteq H^2(Y; \Ar)$ be the (closure of) the ample (nef, K\"ahler) cone of $Y$ in $\mathcal{C}^+$. By definition, $\overline{\mathcal{A}}(Y)$ is closed in $\mathcal{C}^+$ but not in general in $H^2(Y; \Ar)$.
\end{definition}

\begin{definition} Let $\alpha \in H^2(Y; \Zee), \alpha \neq 0$. The \textsl{ oriented wall $W^\alpha$ associated to $\alpha$} is the set $\{x\in \mathcal{C}^+: x\cdot \alpha =0\}$, i.e.\ the intersection of $\mathcal{C}^+$ with the orthogonal space to $\alpha$, together with the preferred half space defined by $x\cdot \alpha \geq 0$. If $C$ is a curve on $Y$, we write $W^C$ for $W^{[C]}$. A standard result (see for example \cite{FriedmanMorgan}, II (1.8)) shows that, if $I$ is a subset of $H^2(Y; \Zee)$ and there exists an $N\in \Zee^+$ such that $-N \leq \alpha^2 < 0$ for all $\alpha \in I$, then the collection of walls $\{W^\alpha: \alpha \in I\}$ is locally finite on $\mathcal{C}^+$. Finally, we say that $W^\alpha$ is a \textsl{face} of $\overline{\mathcal{A}}(Y)$ if $\partial  \overline{\mathcal{A}}(Y)\cap W^\alpha$ contains an open subset of $W^\alpha$ and $x\cdot \alpha \geq 0$ for all $x\in \overline{\mathcal{A}}(Y)$.
\end{definition}

\begin{lemma} $\overline{\mathcal{A}}(Y)$ is the set of all $x\in \mathcal{C}^+$ such that  $x\cdot [D_i]\geq 0$, $x\cdot [E] \geq 0$ for all exceptional curves $E$  and $x\cdot [C] \geq 0$ for all $-2$-curves $C$. Moreover, if   $\alpha$ is the class associated to an exceptional or $-2$-curve, or $\alpha =[D_i]$ for some $i$ such that $D_i^2< 0$  then $W^\alpha$ is a face of $\overline{\mathcal{A}}(Y)$. If $\alpha, \beta$ are two such classes, $W^\alpha = W^\beta$ $\iff$ $\alpha =\beta$. 
\end{lemma}
\begin{proof} For the first claim, it is enough to show that, if $G$ is an irreducible curve on $Y$ with $G^2 < 0$, then $G$ is either $D_i$ for some $i$, an exceptional curve or a $-2$-curve. This follows immediately from adjunction since, if $G\neq D_i$ for any $i$, then $G\cdot D \geq 0$ and $-2 \leq 2p_a(G) -2 = G^2 - G \cdot D < 0$, hence $p_a(G) =0$ and either $G^2 = -2$, $G \cdot D =0$, or $G^2 = G \cdot D =-1$. The last two statements follow from the openness statement in Proposition~\ref{constructnef} and the fact that no two distinct classes of the types listed above are multiples of each other. 
\end{proof}

As an alternate characterization of the classes in the previous lemma, we have:

\begin{lemma}\label{numeric} Let $H$ be a  nef   divisor  such that $H\cdot D >0$.
\begin{enumerate}
\item[\rm(i)] If    $\alpha \in H^2(Y; \Zee)$ with $\alpha ^2 = \alpha \cdot [K_Y] = -1$, then  $\alpha \cdot [H] \geq 0$ $\iff$ $\alpha$ is the class of an effective curve. In particular, the wall $W^\alpha$ does not pass through the interior of $\overline{\mathcal{A}}(Y)$  (cf.\ \cite{FriedmanMorgan}, p.\ 332 for a more general statement).
\item[\rm(ii)] If    $\beta \in H^2(Y; \Zee)$ with $\beta ^2= -2$, $\beta\cdot [D_i] = 0$ for all $i$, $\beta \cdot [H] \geq  0$, and $\varphi_Y(\beta) =1$, then $\pm \beta$ is the class of an effective curve, and $\beta$ is effective if $\beta \cdot [H] > 0$.
\end{enumerate}
Hence the ample cone $\overline{\mathcal{A}}(Y)$ is the set of all $x\in \mathcal{C}^+$ such that  $x\cdot [D_i]\geq 0$ and  $x\cdot\alpha \geq 0$ for all classes $\alpha$ and $\beta$ as described in {\rm(i)} and {\rm(ii)} above, where in case {\rm(ii)} we assume in addition that $\beta$ is effective, or equivalently that $\beta \cdot [H] > 0$ for some nef divisor $H$.
\end{lemma}
\begin{proof}  (i) Clearly, if $\alpha$ is the class of an effective curve, then $\alpha \cdot [H] \geq 0$ since $H$ is nef. Conversely, assume that $\alpha ^2 = \alpha \cdot [K_Y] = -1$ and that $\alpha \cdot [H] \geq 0$.  By Riemann-Roch, $\chi(L_\alpha) = 1$. Hence either $h^0(L_\alpha) > 0$ or $h^2(L_\alpha) > 0$. But $h^2(L_\alpha) =  h^0(L_\alpha^{-1}\otimes K_Y)$ and $[H] \cdot (-\alpha - [D] ) < 0$, by assumption. Thus $h^0(L_\alpha) > 0$ and hence $\alpha$ is the class of an effective curve.

\smallskip
\noindent (ii) As in (i), $H \cdot (-\beta - [D] ) < 0$, and hence $h^0(L_\beta^{-1}\otimes K_Y) = 0$. Thus $h^2(L_\beta) =0$. Suppose that $h^0(L_\beta) = 0$. Then, by Riemann-Roch, $\chi(L_\beta) = 0$ and hence $h^1(L_\beta) = 0$. Hence $h^1(L_\beta^{-1}\otimes K_Y) = 0$. Since $\varphi_Y(\beta) =1$, $L_\beta^{\pm 1}|D = \scrO_D$.  Thus there is an exact sequence
$$0 \to L_\beta^{-1}\otimes \scrO_Y(-D) \to L_\beta^{-1}\to\scrO_D \to 0.$$
Since $H^1(L_\beta^{-1}\otimes K_Y) =H^1(L_\beta^{-1}\otimes \scrO_Y(-D)) =0$, the map $H^0(L_\beta^{-1})\to H^0(\scrO_D)$ is surjective and hence $-\beta$ is the class of an effective curve.
\end{proof}

It is natural to make the following definition:

\begin{definition} Let $\alpha \in H^2(Y; \Zee)$. Then $\alpha$ is a \textsl{numerical exceptional curve} if $\alpha ^2 = \alpha \cdot [K_Y] = -1$. The numerical exceptional curve $\alpha$ is \textsl{effective} if $h^0(L_\alpha) > 0$, i.e.\ if $\alpha = [G]$, where $G$ is an effective curve.
\end{definition}

A minor variation of the   proof of Lemma~\ref{numeric} shows:

\begin{lemma}\label{remarkafternumeric}  Let $H$ be a nef and big divisor such that $H\cdot G > 0$ for all $G$ an irreducible curve not equal to $D_i$ for some $i$, and let $\alpha$ be a numerical exceptional curve.
\begin{enumerate}
\item[\rm{(i)}] Suppose that  $[H]\cdot\alpha  \geq 0$. Then either $[H]\cdot\alpha  > 0$ and $\alpha$ is effective or $H\cdot D =  [H]\cdot\alpha   =0$ and $\alpha$ is an integral linear combination of the $[D_i]$. 
\item[\rm{(ii)}] If $(Y,D)$ is negative definite and $\alpha$ is an integral linear combination of the $[D_i]$, then either some component $D_i$ is a smooth rational curve of self-intersection $-1$ or $K_Y^2=-1$, $\alpha = K_Y$ and hence $\alpha$ is not effective. 
\item[\rm{(iii)}] If no component $D_i$ is a smooth rational curve of self-intersection $-1$, then $\alpha$ is  effective  $\iff$ $[H] \cdot \alpha > 0$.
\end{enumerate}
\end{lemma}
\begin{proof} (i) As in the proof of Lemma~\ref{numeric}, either $\alpha$ or $-\alpha -[D]$ is the class of an effective divisor. If $-\alpha -[D]$ is the class of an effective divisor, then $0\leq [H]\cdot (-\alpha -[D]) \leq 0$, so that $[H]\cdot \alpha=H \cdot D = 0$. In particular $(Y,D)$ is negative definite. Moreover, if $G$ is an effective divisor with $[G] = -\alpha -[D]$, then every component of $G$ is equal to some $D_i$, hence $[G]$ and therefore $\alpha = -[G]-[D]$ are integral linear combinations of the $[D_i]$.

\medskip
\noindent (ii)   Suppose that $\alpha$ is an integral linear combination of the $[D_i]$ but that no $D_i$ is a smooth rational curve of self-intersection $-1$. We shall show that  $K_Y^2=-1$ and $\alpha = K_Y$.  First suppose that $K_Y^2=-1$. Then $\bigoplus_i\Zee\cdot [D_i] = \Zee\cdot [K_Y] \oplus L$, where $L$, the orthogonal complement of $[K_Y]$ in $\bigoplus_i\Zee\cdot [D_i]$, is even and negative definite. Thus $\alpha = a[K_Y] + \beta$, with either $\beta = 0$ or $\beta^2 \leq -2$, and $\alpha ^2 = -a^2 + \beta^2$. Hence, if $\alpha ^2 = \alpha \cdot [K_Y] = -1$, the only possibility is $\beta = 0$ and $a=1$. In case $K_Y^2 < -1$, $D$ is reducible, and no $D_i$ is a smooth rational curve of self-intersection $-1$, then  $D_i^2\leq -2$ for all $i$ and either $D_i^2 \leq -4$ for some $i$ or there exist $i\neq j$ such that $D_i^2 = D_j^2=-3$. In this case, it is easy to check that, for all integers $a_i$ such that $a_i \neq 0$ for some $i$, $(\sum_ia_iD_i)^2 < -1$. This contradicts $\alpha^2 =-1$.

\medskip
\noindent (iii) If $[H]\cdot \alpha > 0$, then $\alpha$ is effective by (i). If $[H]\cdot \alpha< 0$, then clearly $\alpha$ is not effective. Suppose that $[H]\cdot \alpha =0$; we must show that, again, $\alpha$ is not effective. Suppose that $\alpha=[G]$ is effective. By the hypothesis on $H$, every component of $G$ is a $D_i$ for some $i$, so that $\alpha = \sum_ia_i[D_i]$ for some $a_i \in \Zee$, $a_i \geq 0$.  Let $I\subseteq \{1, \dots, r\}$ be the set of $i$ such that $a_i > 0$. Then $H\cdot D_i = 0$ for all $i\in I$. If $I = \{1, \dots, r\}$, then $(Y,D)$ is negative definite and we are done by (ii). Otherwise, $\bigcup_{i\in I}D_i$ is a union of chains of curves whose components $D_i$ satisfy $D_i ^2 \leq -2$. It is then easy to check that $\alpha^2 < -1$ in this case, a contradiction. Hence  $\alpha$ is not effective.
\end{proof}

\begin{definition} Let $Y_t$ be a generic small deformation of $Y$, and identify $H^2( Y_t ;\Ar)$ with $H^2( Y ;\Ar)$. Define $\overline{\mathcal{A}}_{\text{\rm{gen}}}= \overline{\mathcal{A}}_{\text{\rm{gen}}}(Y)$ to be the ample cone $\overline{\mathcal{A}}(Y_t)$ of $Y_t$, viewed as a subset of $H^2( Y ;\Ar)$.
\end{definition}

\begin{lemma}\label{describeA}  With notation as above,
\begin{enumerate} 
\item[\rm(i)] If there do not exist any $-2$-curves on $Y$, then $\overline{\mathcal{A}}(Y) = \overline{\mathcal{A}}_{\text{\rm{gen}}}$. More generally, 
$\overline{\mathcal{A}}_{\text{\rm{gen}}}$ is the set of all $x\in \mathcal{C}^+$ such that $x\cdot [D_i] \geq 0$ and 
$x\cdot \alpha \geq 0$   for all   effective numerical exceptional curves. In particular,
$$\overline{\mathcal{A}}(Y) \subseteq \overline{\mathcal{A}}_{\text{\rm{gen}}}.$$
\item[\rm(ii)] $\overline{\mathcal{A}}(Y) = \{ x\in \overline{\mathcal{A}}_{\text{\rm{gen}}}: x\cdot [C] \geq 0 \text{ for all $-2$-curves $C$} \}$.
\end{enumerate}
\end{lemma}
\begin{proof} Let $Y$ be  a surface with   no $-2$-curves (such surfaces exist and are generic by the surjectivity of the period map, Theorem~\ref{surjper}). Fix a nef   divisor $H$ on $Y$ with $H\cdot D >0$. Then $\overline{\mathcal{A}}(Y)$ is the set of all $x\in \mathcal{C}^+$ such that $x\cdot [D_i] \geq 0$ and $x\cdot [E] \geq 0$ for all exceptional curves $E$, and this last condition is equivalent to $x\cdot \alpha \geq 0$   for all $\alpha\in H^2(Y; \Zee)$ such that $\alpha^2 =\alpha\cdot [K_Y] =-1$ and $\alpha \cdot [H] \geq 0$, by Lemma~\ref{numeric}. Since this condition is independent of the choice of $Y$, because we can choose the divisor $H$ to be ample and to vary in a small deformation, the first part of (i) follows, and the remaining statements are clear.
\end{proof}

In fact, the argument above shows:

\begin{lemma}\label{definv} The set of effective numerical exceptional curves and the set $\overline{\mathcal{A}}_{\text{\rm{gen}}}$ are  locally constant, and hence are invariant in a global deformation with trivial monodromy under the induced identifications. \qed
\end{lemma} 

\begin{lemma}\label{reflect}  If $C$ is a $-2$-curve on $Y$, then the wall $W^C$ meets the interior of $\overline{\mathcal{A}}_{\text{\rm{gen}}}$, and in fact $r_C( \overline{\mathcal{A}}_{\text{\rm{gen}}}) = \overline{\mathcal{A}}_{\text{\rm{gen}}}$, where $r_C\colon H^2(Y; \Ar) \to H^2(Y; \Ar)$ is reflection in the class $[C]$.
 Hence $\overline{\mathcal{A}}(Y)$ is a fundamental domain for the action of the  group $\mathsf{W}({\Delta_Y})$ on $\overline{\mathcal{A}}_{\text{\rm{gen}}}$, where $\mathsf{W}({\Delta_Y})$ is the group  generated by the reflections in the classes in the set  $\Delta_Y$ of $-2$-curves on $Y$. 
\end{lemma}
\begin{proof}  Clearly, if $r_C( \overline{\mathcal{A}}_{\text{\rm{gen}}}) = \overline{\mathcal{A}}_{\text{\rm{gen}}}$, then $W^C$ meets the interior of $\overline{\mathcal{A}}_{\text{\rm{gen}}}$.
To see that $r_C( \overline{\mathcal{A}}_{\text{\rm{gen}}}) = \overline{\mathcal{A}}_{\text{\rm{gen}}}$, assume first more generally that $\beta\in \Lambda$ is any class with $\beta^2 = -2$, and let $r_\beta$ be the corresponding reflection. Then $r_\beta$ permutes the set of $\alpha \in H^2(Y; \Zee)$ such that $\alpha^2 =\alpha\cdot [K_Y] =-1$, but does not necessarily preserve the condition that $\alpha$ is effective, i.e.\ that $\alpha \cdot [H] \geq 0$ for some nef   divisor $H$ on $Y$ with $H\cdot D >0$. However, for $\beta = [C]$, there exists by Proposition~\ref{constructnef} a nef  and big divisor  $H_0$ such that $H_0 \cdot C = 0$ and $H\cdot D > 0$. Hence  $[H_0]$ is invariant under $r_C$,   and so $r_C$ permutes the set of $\alpha \in H^2(Y; \Zee)$ such that $\alpha^2 =\alpha\cdot [K_Y] =-1$ and   $\alpha \cdot [H_0] \geq 0$. Thus $r_C$ permutes the set of effective numerical exceptional curves and hence the faces of $\overline{\mathcal{A}}_{\text{\rm{gen}}}$, so that $r_C( \overline{\mathcal{A}}_{\text{\rm{gen}}}) = \overline{\mathcal{A}}_{\text{\rm{gen}}}$. Since $\overline{\mathcal{A}}(Y) \subseteq \overline{\mathcal{A}}_{\text{\rm{gen}}}$ is given by (ii) of Lemma~\ref{describeA}, the final statement is then a general result in the theory of reflection groups (cf.\ \cite{Bour}, V \S3).
\end{proof}

\begin{remark}\label{monodromyinvar} (i) The argument for the first part of Lemma~\ref{reflect} essentially boils down to the following: let $\overline{Y}$ be the normal surface obtained by contracting $C$. Then the reflection $r_C$ is the monodromy associated to a generic smoothing of the singular surface $\overline{Y}$, and the cone $\overline{\mathcal{A}}_{\text{\rm{gen}}}$ is invariant under monodromy.

\smallskip
\noindent (ii) If $E$ is an exceptional curve, then $W^E$  is a face of $\overline{\mathcal{A}}(Y)$. For a generic $Y$ (i.e.\ no $-2$-curves), Lemma~\ref{reflect} then says that the set of exceptional curves on $Y$ is invariant under the reflection group generated by all classes of square $-2$ which become the classes of a $-2$-curve under some specialization. A somewhat more involved statement holds in the nongeneric case.
\end{remark}

\begin{lemma}\label{permapinvar} With $\mathsf{W}({\Delta_Y})$ as in Definition~\ref{defcurves},  for all $w\in \mathsf{W}({\Delta_Y})$ and all $\beta \in \Lambda$,  $\varphi_{Y}(w(\alpha)) = \varphi_{Y}(\alpha)$. 
\end{lemma}
\begin{proof} This is clear since $\varphi_Y([C]) =1$, hence $\varphi_{Y}(r_C(\alpha)) = \varphi_{Y}(\alpha)$ for all $\alpha \in \Lambda$.
\end{proof}

\begin{lemma}\label{Weyltrans} Suppose that $C=\sum_ia_iC_i$, where the $C_i$ are $-2$-curves, $a_i\in \Zee$, $C^2 = -2$, the support of $C$ is connected, and $(C_i\cdot C_j)$ is negative definite. Then there exists an element $w$ in the  group generated by reflections in the $[C_i]$ such that $w([C]) = [C_i]$ for some $i$.
\end{lemma}
\begin{proof} This follows from the well known fact that, if $R$ is an irreducible root system such that all roots have the same length, then the Weyl group $\mathsf{W}(R)$ acts transitively on the set of roots.
\end{proof}

\begin{theorem}\label{mainprop}  Let $\beta \in \Lambda$ with $\beta^2 = -2$. Then the following are equivalent:
\begin{enumerate}
\item[\rm(i)] Let $Y_1$ be  a deformation of $Y$ with trivial monodromy such that $\varphi_{Y_1}(\beta) = 1$. Then, with  $\mathsf{W}({\Delta_{Y_1}})$ as in Definition~\ref{defcurves}, there exists   $w\in \mathsf{W}({\Delta_{Y_1}})$  such that $w(\beta)=[C]$, where   $C$ is a $-2$-curve  on $Y_1$. In particular, if $Y_1$ is generic subject to the condition that $\varphi_{Y_1}(\beta) = 1$ (i.e.\ if $\Ker  \varphi_{Y_1}  = \Zee \cdot \beta$), then $\pm \beta = [C]$ for a $-2$-curve $C$.
\item[\rm(ii)] The wall $W^\beta$ meets the interior of $\overline{\mathcal{A}}_{\text{\rm{gen}}}$.
\item[\rm(iii)] If $r_\beta$ is reflection in the class $\beta$, then $r_\beta(\overline{\mathcal{A}}_{\text{\rm{gen}}}) = \overline{\mathcal{A}}_{\text{\rm{gen}}}$.
\end{enumerate}
\end{theorem}
\begin{proof} Lemma~\ref{reflect} implies that (i) $\implies$  (iii) in case $Y=Y_1$ and $\beta = [C]$ where   $C$ is a $-2$-curve. The   case  where $w(\beta) = [C]$ follows easily from this since, for all $w\in \mathsf{W}({\Delta_{Y_1}})$, $w\circ r_\beta \circ w^{-1} = r_{w(\beta)}$. Lemma~\ref{definv} then handles the case where $Y_1$ is replaced by a general deformation $Y$. Also, clearly (iii) $\implies$ (ii). So it is enough to show that (ii) $\implies$ (i). In fact, by Lemma~\ref{Weyltrans}, it is enough to show  that, if $Y$ is any surface such that  $\varphi_Y(\beta) = 1$ and  $W^\beta$ meets the interior of $\overline{\mathcal{A}}_{\text{\rm{gen}}}$, then there exists a $w\in \mathsf{W}({\Delta_Y})$ such that $w(\beta) = [\sum_ia_iC_i]$ where $a_i\in \Zee^+$, the $C_i$ are  curves disjoint from $D$, and $\bigcup_iC_i$ is connected. 

By hypothesis, there exists an $x$ in the interior of $\overline{\mathcal{A}}_{\text{\rm{gen}}}$ such that $x\cdot \beta =0$. In particular, $x\cdot [D_i] >0$ for all $i$. We can assume that $x=[H]$ is the class of a divisor $H$. After replacing $x$ by $w(x)$ and $\beta$ by $w(\beta)$ for some $w\in \mathsf{W}({\Delta_Y})$, we can assume that $x$ (and hence $H$) lies in $\overline{\mathcal{A}}(Y)$, so that $H$ is a nef and big divisor with $H\cdot D_i > 0$ for all $i$, and we still have $\varphi_Y(\beta) = 1$ by Lemma~\ref{permapinvar}. By Lemma~\ref{numeric}, possibly after replacing $\beta$ by $-\beta$, $\beta = [\sum_ia_iC_i]$ where the $C_i$ are irreducible curves and $a_i \in \Zee^+$. Since $\beta\cdot [H] =\sum_ia_i(C_i\cdot H)=0$, $C_i\cdot H \geq 0$, and $D_j\cdot H > 0$, $C_i\cdot H = 0$ for all $i$ and no $C_i$ is equal to $D_j$ for any $j$. Hence the  $C_i$ are curves meeting each $D_j$ in at most finitely many points and $\sum_ia_i(C_i\cdot D_j)=0$,   so that $C_i\cap D_j =\emptyset$. Finally each $(C_i)^2 < 0$ by Hodge index, and so each $C_i$ is a $-2$-curve. Moreover the $C_i$ span a negative definite lattice, and in particular their classes are independent. From this, the statement about the connectedness of $\bigcup_iC_i$ is clear. 
\end{proof}

\begin{definition} Let $R=R_Y$ be the set  of all $\beta \in \Lambda$ such that $\beta ^2 = -2$ and such that there exists some deformation of $Y$ for which $\beta$ becomes the class of a $-2$-curve. Following \cite{GHK}, we call $R$ the set of \textsl{Looijenga roots} (or briefly \textsl{roots}) of $Y$. Note that $R$ only depends on the deformation type of $Y$.

The definition of $R$ is slightly ill-posed, since we have not specified an identification of the cohomologies of the fibers along the deformation. In particular, if $\beta = [C]$ is a $-2$-curve on $Y$, then by (i) of Remark~\ref{monodromyinvar}, if $Y'$ is a nearby deformation of $Y$, then a general smoothing of the ordinary double point on the contraction of $C$ on $Y$ has monodromy which sends $[C]$ to $-[C]$, and hence $-\beta \in R$ as well. To avoid this issue, it is simpler to define $R$ to be the set of $\beta \in \Lambda$, $\beta^2=-2$,  which satisfy either of the equivalent conditions (ii), (iii) of Theorem~\ref{mainprop}.

Given $Y$, let $\Delta_Y$ be the set of classes of $-2$-curves on $Y$ and $\mathsf{W}({\Delta_Y})$ the reflection group generated by $\Delta_Y$. Finally set $R^{\text{\rm{nod}}}$, the set of \textsl{nodal classes}, to be $\mathsf{W}({\Delta_Y})\cdot \Delta_Y$. Then $R^{\text{\rm{nod}}} \subseteq R$. 
\end{definition}

\begin{corollary}\label{preserveamp} \text{\rm{(i)}} If $f\colon H^2(Y; \Zee) \to H^2(Y; \Zee)$ is an integral isometry preserving the classes $[D_i]$ such that $f(\overline{\mathcal{A}}_{\text{\rm{gen}}}) = \overline{\mathcal{A}}_{\text{\rm{gen}}}$, then $f(R) = R$.

\smallskip
\noindent \text{\rm{(ii)}} If $\mathsf{W}(R)$ is the reflection group generated by reflections in the elements of $R$, then $\mathsf{W}(R) \cdot R = R$ and $w(\overline{\mathcal{A}}_{\text{\rm{gen}}}) = \overline{\mathcal{A}}_{\text{\rm{gen}}}$ for all $w\in \mathsf{W}(R)$.
\qed
\end{corollary}

\begin{remark} A  result similar to Theorem~\ref{mainprop} classifies the elements of $H^2(Y;\Zee)$ which are represented by the class of a smoothly embedded $2$-sphere of self-inter\-section $-2$ in terms of the ``super $P$-cell" of \cite{FriedmanMorgan}.
\end{remark}

In \cite{Looij}, for the case where the length $r(D) \leq 5$, Looijenga defines a subset $R_L$ of  $\Lambda$ by starting with a particular configuration $B$ of elements of square $-2$ (a \textsl{root basis} in the terminology of \cite{Looij}), and setting $R_L = \mathsf{W}(B)\cdot B$, where $\mathsf{W}(B)$ is the reflection group generated by $B$. In fact, the set $R_L$ is just the set $R$ of Looijenga roots:

\begin{proposition} In the above notation, $R_L = R$.
\end{proposition}
\begin{proof} It is easy to see from the construction of \cite[I \S2]{Looij} that $B \subseteq R$. Hence $R_L \subseteq R$. Conversely, if $\alpha \in R$, then, by (ii) of Corollary~\ref{preserveamp}, $r_\alpha(\overline{\mathcal{A}}_{\text{\rm{gen}}}) = \overline{\mathcal{A}}_{\text{\rm{gen}}}$. It then follows from \cite[Proposition I (4.7)]{Looij} that $r_\alpha \in \mathsf{W}(B)$. By a general result in the theory of reflection groups \cite[V \S3.2, Thm.\ 1(iv)]{Bour}, $r_\alpha = r_\beta$ for some $\beta \in R_L$. Thus $\alpha =\pm \beta$, so that $\alpha \in R_L$. Hence $R\subseteq R_L$, and therefore $R_L = R$.
\end{proof} 

\begin{example}\label{irredex} Let $(Y,D)$ be the blowup of $\Pee^2$ at $N \geq 10$  points on an irreducible  nodal cubic curve. We let $h$ be the pullback of the class of a line on $\Pee^2$ and $e_1, \dots, e_N$ be the classes of the exceptional curves.

\smallskip
\noindent (i) Let $\alpha = -3h + \sum_{i=1}^{10}e_i$. Then $\alpha^2 = \alpha \cdot [K_Y] = -1$, so that $\alpha$ is a numerical exceptional curve. But there exists a nef and big divisor $H$ (for example $h$) such that $\alpha \cdot [H] < 0$, so that $\alpha$ is not effective. Hence, $\alpha \cdot x \leq 0$ for all $x\in  \overline{\mathcal{A}}(Y)=\overline{\mathcal{A}}_{\text{\rm{gen}}}$, since $W^\alpha$ does not pass through the interior of $\overline{\mathcal{A}}_{\text{\rm{gen}}}$. Note that $W^\alpha$ is never a face of $\overline{\mathcal{A}}_{\text{\rm{gen}}}$. For $N=10$, $W^{-\alpha}$ is   a face of $\overline{\mathcal{A}}_{\text{\rm{gen}}}$, but this is no longer the case  for $N \geq 11$. Thus the condition $\alpha \cdot [H] \geq  0$ for some $H$ such that $H\cdot D > 0$ is necessary for $\alpha$ to be effective.

More generally, let $f = 3h -\sum_{i=1}^9e_i$ and set $\alpha = kf + e_{10}$ (the case above corresponds to $k=-1$). As above, $\alpha$ is a numerical exceptional curve. For $k\leq -1$, $h\cdot \alpha < 0$, and hence $\alpha$ is not effective. For $k\geq 1$, $\alpha$ is effective but it is not the class of an exceptional curve: for all $x\in \overline{\mathcal{A}}_{\text{\rm{gen}}}$, $x\cdot f > 0$, and $x\cdot e_{10}\geq 0$. Hence $x\cdot \alpha > 0$   for all $x \in \overline{\mathcal{A}}_{\text{\rm{gen}}}$. Thus  $W^\alpha$ is not a face of $\overline{\mathcal{A}}_{\text{\rm{gen}}}$ and so $\alpha$ is not the class of an exceptional curve. 

\smallskip 
\noindent  (ii) With $\alpha$ any of the classes as above, suppose that $N \geq 11$ and $k\neq 0$ and set $\beta = \alpha -e_{11}$. Then $\beta^2 =-2$ and $\beta \cdot [K_Y] = 0$. However, 
$$r_\beta(e_{11}) = e_{11} + (e_{11}\cdot \beta)\beta  = \alpha.$$
Since $W^{e_{11}}$ is a face of $\overline{\mathcal{A}}_{\text{\rm{gen}}}$ and $W^\alpha$ is not a face of $\overline{\mathcal{A}}_{\text{\rm{gen}}}$, $r_\beta( \overline{\mathcal{A}}_{\text{\rm{gen}}})\neq \overline{\mathcal{A}}_{\text{\rm{gen}}}$. Hence $\beta$ does not satisfy any of the equivalent conditions of Theorem~\ref{mainprop}, so that $\beta \notin R$.
\end{example}

\begin{remark}\label{bestposs} In the situation of the example above, it is well-known that if $D$ is irreducible,  $N \leq 9$ (i.e.\ $D^2\geq 0$), and there are no $-2$-curves on $Y$, then every numerical exceptional curve is the class of an exceptional curve, so (i) above is best possible. A generalization is given in Proposition~\ref{nonneg} below. We shall show in Proposition~\ref{minusone} that the example in (ii) is best possible as well.
\end{remark}

The numerical exceptional curves given in (i) of Example~\ref{irredex} were known to Du Val. In fact, he showed that they are essentially the only numerical curves in case $Y$ is the blowup of $\Pee^2$ at $10$ points (\cite{duV2}, pp.\ 46--47):

\begin{proposition} Suppose that $(Y,D)$ is the blowup of $\Pee^2$ at $10$ points  lying on an irreducible cubic, that $Y$ is generic in the sense that there are no $-2$-curves on $Y$,  and that $\alpha$  is a numerical exceptional curve. Then there exists an exceptional curve $E$ on $Y$ and an integer $k$ such that $\alpha$ is the class of $k(D + E) + E$.
\end{proposition}
\begin{proof}   Suppose that $\alpha$ is a numerical exceptional curve on $Y$. Then, since $K_Y^2=-1$,  $\lambda = \alpha +[D] = \alpha -[K_Y]$ satisfies: $\lambda^2 = \lambda \cdot \alpha = \lambda \cdot [K_Y] = 0$. In particular, $\lambda \in \Lambda$. Conversely, given an isotropic vector $\lambda \in \Lambda$, if we set $\alpha = \lambda + [K_Y]$, then $\alpha$ is a numerical exceptional curve. 

Any isotropic vector $\lambda \in \Lambda$ can be uniquely written as $n\lambda_0$, where $n \in \Zee$ and $\lambda_0$ is primitive and lies in $\overline{\mathcal{C}^+}$.  Note that $H^2(Y; \Zee) = \Zee[K_Y] \oplus \Lambda$ and that $\Lambda = U \oplus (-E_8)$ (both sums orthogonal). An easy exercise shows that, if $\Aut^+(\Lambda)$ is the group of integral isometries $A$ of $\Lambda$ such that $A(\mathcal{C}^+\cap \Lambda) = \mathcal{C}^+\cap \Lambda$, i.e.\ $A$ has real spinor norm equal to $1$, then every $A \in \Aut^+(\Lambda)$ extends uniquely to an integral isometry of $H^2(Y; \Zee)$ fixing $[K_Y]$ and hence $[D]$, and moreover that $\Aut^+(\Lambda)$ acts transitively on the set of (nonzero) primitive isotropic vectors in $\overline{\mathcal{C}^+} \cap \Lambda$.  Hence there exists an $A \in \Aut^+(\Lambda)$ such that $A(\lambda_0) = f$, in the notation of Example~\ref{irredex}. If we continue to denote by $A$ the extension of $A$ to an isometry   of $H^2(Y; \Zee)$, then $A(\alpha) = nf + [K_Y] = (n-1) f + e_{10}$, since $f =  -[K_Y] + e_{10}$. It follows that $\alpha = (n-1)\lambda_0 + A^{-1}(e_{10})$. Using Proposition~\ref{minusone} below, $A^{-1}$ preserves the walls of the ample cone of $Y$, and thus $A^{-1}(e_{10}) =e$ is the class of an exceptional curve $E$, and $\lambda_0 =A^{-1}(f) = A^{-1}([D] + e_{10}) = [D] + E$. Hence, setting $k=n-1$,  $\alpha$ is the class of $k(D + E) + E$ as claimed.
\end{proof}

The proof above shows the following:

\begin{corollary} Let $(Y,D)$ be the blowup of $\Pee^2$ at $10$ points  lying on an irreducible cubic  and such that there are no $-2$-curves on $Y$,   let $\alpha$ be a numerical exceptional curve on $Y$, and let $\lambda = \alpha -[K_Y]$. Then:
\begin{enumerate}
\item[\rm{(i)}] $\alpha$ is effective $\iff$ $\lambda \in  (\overline{\mathcal{C}^+}-\{0\}) \cap \Lambda$.
\item[\rm{(ii)}] $\alpha$ is not effective $\iff$ $\lambda \in (-\overline{\mathcal{C}^+}) \cap \Lambda$.
\item[\rm{(iii)}] $\alpha$ is the class of an exceptional curve $\iff$ $\lambda$ is a  primitive isotropic vector in $\overline{\mathcal{C}^+}  \cap \Lambda$. Thus there is a bijection from the set of exceptional curves on $Y$ to the set of  primitive isotropic vectors in $\overline{\mathcal{C}^+}  \cap \Lambda$. \qed
\end{enumerate}
\end{corollary}

\begin{remark} In the above situation, let  $\mathsf{W}$ be the group generated by the reflections in the classes $e_1-e_2, \dots, e_9-e_{10}, h-e_2-e_2-e_3$, which are easily seen to be Looijenga roots. A classical argument (usually called Noether's inequality) shows that, if $\lambda_0$ is a primitive integral isotropic vector  in $\Lambda$ lying in $\overline{\mathcal{C}^+}$, then  there exists $w\in \mathsf{W}$ such that $w(\lambda_0) = f = 3h -\sum_{i=1}^9 e_i$, in the notation of Example~\ref{irredex}.  Thus, $\mathsf{W}$ acts transitively on the set of such vectors. Using  standard results about the affine Weyl group of $E_8$, it is then easy to see that $\mathsf{W} = \Aut^+(\Lambda)$.  This was already noted by Du Val in \cite{duV2}.
\end{remark}

\section{Roots and the ample cone}

By Corollary~\ref{preserveamp}, if $f\colon H^2(Y; \Zee) \to H^2(Y; \Zee)$ is an integral isometry preserving the classes $[D_i]$ such that $f(\overline{\mathcal{A}}_{\text{\rm{gen}}}) = \overline{\mathcal{A}}_{\text{\rm{gen}}}$, then $f(R) = R$. In this section, we find criteria for when the converse holds. We begin with the following:

\begin{lemma}\label{containopen} Let $f\colon H^2(Y; \Zee) \to H^2(Y; \Zee)$ be an integral isometry preserving $\mathcal{C}^+$ and the classes $[D_i]$. If $f(\overline{\mathcal{A}}_{\text{\rm{gen}}}) \cap \overline{\mathcal{A}}_{\text{\rm{gen}}}$ contains an open set, then $f(\overline{\mathcal{A}}_{\text{\rm{gen}}}) = \overline{\mathcal{A}}_{\text{\rm{gen}}}$.
\end{lemma} 
\begin{proof} Choosing $x\in f(\overline{\mathcal{A}}_{\text{\rm{gen}}}) \cap \overline{\mathcal{A}}_{\text{\rm{gen}}}$ corresponding to an ample divisor, it is easy to see that $f(\overline{\mathcal{A}}_{\text{\rm{gen}}})$ and $\overline{\mathcal{A}}_{\text{\rm{gen}}}$ have the same set of walls, hence are equal.
\end{proof}

Next we deal with the case where one component of $D$ is a smooth rational curve of self-intersection $-1$. 

\begin{lemma}\label{excepcomp} Suppose that $D$ is reducible and that $D_r^2=-1$. Let $(\overline{Y}, \overline{D})$ be the anticanonical pair obtained by contracting $D_r$. Then any isometry $f$ of $H^2(Y;\Zee)$ preserving the classes $[D_i]$, $1\leq i\leq r$, defines an isometry $\bar{f}$ of $H^2(\overline{Y};\Zee)$ preserving the classes $[\overline{D}_i]$, $1\leq i\leq r-1$, and conversely. Moreover, $f$ preserves $\overline{\mathcal{A}}_{\text{\rm{gen}}}(Y)$ $\iff$ $\bar{f}$ preserves   $\overline{\mathcal{A}}_{\text{\rm{gen}}}(\overline{Y})$, and  $R_Y$ is naturally identified with the roots  $R_{\overline{Y}}$ of $\overline{Y}$.
\end{lemma}
\begin{proof} The first statement is clear. Identifying $H^2(\overline{Y}, \Zee)$ with $[D_r]^\perp \subseteq H^2(Y; \Zee)$, it is clear that $\overline{\mathcal{A}}_{\text{\rm{gen}}}(Y) \cap [D_r]^\perp =\overline{\mathcal{A}}_{\text{\rm{gen}}}(\overline{Y})$. Hence, if $f$ preserves $\overline{\mathcal{A}}_{\text{\rm{gen}}}(Y)$, then  $\bar{f}$ preserves $\overline{\mathcal{A}}_{\text{\rm{gen}}}(\overline{Y})$. Since a divisor $\overline{H}$ on $\overline{Y}$ is ample $\iff$ $N \overline{H} - D_r$ is ample for all $N \gg 0$, it follows that, if  $\bar{f}$ preserves $\overline{\mathcal{A}}_{\text{\rm{gen}}}(\overline{Y})$, then $f(\overline{\mathcal{A}}_{\text{\rm{gen}}}(Y)) \cap \overline{\mathcal{A}}_{\text{\rm{gen}}}(Y)$ contains an open set, and hence $f(\overline{\mathcal{A}}_{\text{\rm{gen}}}(Y)) = \overline{\mathcal{A}}_{\text{\rm{gen}}}(Y)$ by Lemma~\ref{containopen}.  It follows from this and from Theorem~\ref{mainprop} that $R_Y$ is naturally identified with   $R_{\overline{Y}}$ (or directly from the definition by noting that there is a bijection from the set of deformations of $(Y,D)$ to those of $(\overline{Y}, \overline{D})$). 
\end{proof}

Henceforth, then, we shall always assume if need be that  no component of $D$ is a smooth rational curve of self-intersection $-1$.

We turn  to the   straightforward case where $(Y,D)$ is not negative definite:

\begin{proposition}\label{nonneg} Suppose that $(Y,D)$ and $(Y', D')$ are two anticanonical pairs  with $r(D) = r(D')$ and  such that  neither pair is  negative definite.  If $f\colon H^2(Y; \Zee) \to H^2(Y'; \Zee)$ is an integral isometry with  $f([D_i]) = [D_i']$ for all $i$, then $f(\overline{\mathcal{A}}_{\text{\rm{gen}}}(Y)) = \overline{\mathcal{A}}_{\text{\rm{gen}}}(Y')$ and hence $f(R_Y) = R_{Y'}$. Moreover, 
$$R_Y = \{\beta \in \Lambda(Y,D): \beta^2 =-2\}.$$
\end{proposition}
\begin{proof}  By Lemma~\ref{excepcomp}, we may assume that no $D_i$ has self-intersection $-1$. The statement that the cycle is not negative definite is then equivalent to the statement that either $D_j^2 \geq 0$ for some $j$ or $D_i^2 = -2$ for all $i$ and $r\geq 2$. In the first case, $D_j$ is nef and $D_j \cdot D > 0$. Hence, if $\alpha$ is a numerical exceptional curve such that $\alpha\cdot [D_i] \geq 0$, then $\alpha$ is effective by Lemma~\ref{numeric}. Thus 
$\overline{\mathcal{A}}_{\text{\rm{gen}}}(Y)$ is the set of all $x\in \mathcal{C}^+_D(Y)$ such that $x\cdot \alpha \geq 0$  for all numerical exceptional curves $\alpha$ such that $\alpha\cdot [D_i] \geq 0$. Since $f(\alpha)^2 = \alpha^2$, $f([D_i]) = [D_i']$,  and $f(\alpha) \cdot [K_{Y'}] = \alpha \cdot [K_Y]$, it follows that $f(\overline{\mathcal{A}}_{\text{\rm{gen}}}(Y)) = \overline{\mathcal{A}}_{\text{\rm{gen}}}(Y')$. Applying this to reflection in a class $\beta$ of square $-2$ in $\Lambda(Y,D)$ then implies that $\beta \in R_Y$.

The case where $D_i^2 = -2$ for every $i$ is similar, using the nef divisor $D = \sum_iD_i$ with $D^2 = 0$. If $\alpha$ is a numerical exceptional curve, then $\alpha$ is effective  since $(-\alpha + [K_Y]) \cdot [D] =  \alpha \cdot [K_Y] = -1$. The rest of the argument proceeds as before.   
\end{proof}

\begin{remark} If $D$ is irreducible and not negative definite (i.e.\ $D^2 \geq 0$) and there are no $-2$-curves on $Y$, then, as is well-known and noted in Remark~\ref{bestposs}, every numerical exceptional curve is the class of an exceptional curve. However, if $D$ is reducible but not negative definite, then, even if there are no $-2$-curves on $Y$, there may well exist numerical exceptional curves which are not effective, and effective numerical exceptional curves which are not the class of an exceptional curve.
\end{remark} 

From now on we assume that $D$ is negative definite. The case $K_Y^2=-1$ can also be handled by straightforward methods, as noted in \cite{Looij}. (See also \cite{FriedmanMorgan}, II(2.7)(c) in case $D$ is irreducible.)

\begin{proposition}\label{minusone} Let $(Y,D)$ and $(Y', D')$ be two negative definite anticanonical pairs with $r(D) = r(D')$ and $K_Y^2 = K_{Y'}^2 = -1$.  Let $f\colon H^2(Y; \Zee) \to H^2(Y'; \Zee)$ be an isometry such that $f([D_i]) = [D_i']$ for all $i$ and  $f (\mathcal{C}^+(Y)) =  \mathcal{C}^+(Y')$. Then $f(\overline{\mathcal{A}}_{\text{\rm{gen}}}(Y)) = \overline{\mathcal{A}}_{\text{\rm{gen}}}(Y')$. Moreover, 
$$R_Y = \{\beta \in \Lambda(Y,D): \beta^2 = -2\},$$
and hence $f(R_Y) = R_{Y'}'$.
\end{proposition}
\begin{proof} Since $(Y,D)$ is negative definite,   no component of $D$ is a smooth rational curve of self-intersection $-1$. Fix a nef and big divisor $H$ such that $H\cdot D_i =0$ for all $i$ and $H\cdot G > 0$ for every irreducible curve $G \neq D_i$. If $\alpha$ is a numerical exceptional curve,   $(\alpha -[K_Y])^2 = (\alpha + [D]) ^2 = 0$. By Lemma~\ref{remarkafternumeric}, $\alpha$ is effective   $\iff$ $[H] \cdot \alpha > 0$ $\iff$ $[H] \cdot (\alpha + [D]) > 0$. By the Light Cone Lemma (cf.\ \cite{FriedmanMorgan}, p.\ 320), this last condition is equivalent to: $\alpha + [D] \in \overline{\mathcal{C}^+}-\{0\}$. Since this condition is clearly preserved by an isometry $f$ as in the statement of the proposition, we see that $f(\overline{\mathcal{A}}_{\text{\rm{gen}}}(Y)) = \overline{\mathcal{A}}_{\text{\rm{gen}}}(Y')$. The final statement then follows as in the proof of Proposition~\ref{nonneg}.   
\end{proof}

\begin{remark} The hypothesis $K_Y^2 =-1$ implies that $r(D) \leq 10$, so there are only finitely many examples of the above type. For $r(D) = 10$, there is essentially just one combinatorial possibility for $(Y,D)$ neglecting the orientation (cf.\ \cite{FriedmanMiranda}, (4.7), where it is easy to check that this is the only possibility). For $r(D) = 9$, however, there are two different possibilities for the combinatorial type of $(Y,D)$ (again ignoring the orientation). Begin with an anticanonical pair $(\overline{Y}, \overline{D})$, where $\overline{Y}$ is a rational elliptic surface and $\overline{D}=\overline{D}_0 + \cdots + \overline{D}_8$ is a fiber of type $\widetilde{A}_8$ (or $I_9$ in Kodaira's notation). There is a unique such rational elliptic surface $\overline{Y}$ and its Mordell-Weil group has order $3$ (see for example \cite{MirandaPersson}). In particular, possibly after relabeling the components, there is an exceptional curve meeting $\overline{D}_i$ $\iff$ $i=0,3,6$. It is easy to see that blowing up a point on a component $\overline{D}_i$ meeting an exceptional curve leads to a different combinatorial possibility for an anticanonical pair $(Y,D)$ with $K_Y^2 =-1$ and $r(D) = 9$ than blowing up a point on a component $\overline{D}_i$ which does not meet an exceptional curve.
\end{remark}

We turn now to the case where $(Y,D)$ is negative definite but with no assumption on $K_Y^2$.

\begin{definition} A point $x\in \mathcal{C}^+ \cap \Lambda$ is \textsl{$R$-distinguished} if there 
exists a codimension one negative definite subspace $V$ of $\Lambda \otimes \Ar$ spanned by elements of $R$ such that $x\in V^\perp$. Note that the definition only depends on the deformation type of the pair $(Y,D)$.
\end{definition}

\begin{remark} Clearly, if $V$ is a codimension one negative definite subspace   of $\Lambda \otimes \Ar$ spanned by elements of $R$, then $V$ is defined over $\Q$ and   $V^\perp \cap (\Lambda \otimes \Ar)$ is a one-dimensional subspace of $H^2(Y; \Ar)$ defined over $\Q$ and spanned by an $h\in H^2(Y; \Zee)$ with $h^2 > 0$, $h\cdot [D_i] =0$, and $h\cdot \beta = 0$ for all   $\beta \in R\cap V$. Hence, if $h\in \mathcal{C}^+\cap \Lambda$, then $h$ is $R$-distinguished.

Also, if the rank of $\Lambda$ is one, then $\{0\}$ is a codimension one negative definite subspace of $\Lambda \otimes \Ar$, and hence every point of $\mathcal{C}^+ \cap \Lambda$ is  $R$-distinguished. 

However, as we shall see, there exist deformation types $(Y,D)$ with no $R$-distinguished points.
\end{remark}

The following is also clear:

\begin{lemma} Let $(Y,D)$ and $(Y', D')$ be two anticanonical pairs with $r(D) = r(D')$ and let $f\colon H^2(Y; \Zee) \to H^2(Y'; \Zee)$ be an isometry such that $f([D_i]) = [D_i']$ for all $i$, $f (\mathcal{C}^+(Y)) =  \mathcal{C}^+(Y')$, and $f(R_Y) = R_{Y'}$. Then, if $x$ is a $R_Y$-distinguished point of $\mathcal{C}^+(Y) \cap \Lambda(Y,D)$, $f(x)$ is a $R_{Y'}$-distinguished point of $ \mathcal{C}^+(Y') \cap \Lambda(Y',D')$. \qed
\end{lemma}

Our goal now is to prove:

\begin{theorem}\label{disttheorem} Let $(Y,D)$ and $(Y', D')$ be two anticanonical pairs with $r(D) = r(D')$ and let $f\colon H^2(Y; \Zee) \to H^2(Y'; \Zee)$ be an isometry such that $f([D_i]) = [D_i']$ for all $i$, $f (\mathcal{C}^+(Y)) =  \mathcal{C}^+(Y')$, and $f(R_Y) = R_{Y'}$.  If there exists a $R$-distinguished point of $\mathcal{C}^+ \cap \Lambda$, then $f(\overline{\mathcal{A}}_{\text{\rm{gen}}}(Y)) = \overline{\mathcal{A}}_{\text{\rm{gen}}}(Y')$.
\end{theorem}

We begin by showing:

\begin{proposition}\label{aprop} Let $x$ be a $R$-distinguished point of $\mathcal{C}^+ \cap \Lambda$. Then $x\in \overline{\mathcal{A}}_{\text{\rm{gen}}}$. Moreover, if $\alpha$ is a numerical exceptional curve and $\alpha$ is not in the span of the $[D_j]$, then $\alpha$ is  effective  $\iff$ $\alpha \cdot x \geq 0$.
\end{proposition}
\begin{proof} It is enough by Lemma~\ref{definv} to check this on some (global) deformation of $(Y,D)$ with trivial monodromy. By Theorem~\ref{surjper}, we can assume that $$\Ker \varphi_Y = V \cap \Lambda,$$ where $V$ is as in the definition of $R$-distinguished.  In particular, if $C \in \Delta_Y$, i.e.\ $C$ is a $-2$-curve on $Y$, then $[C] \in V$. It follows from (i) of Theorem~\ref{mainprop} that every $\beta \in V\cap R$ is a sum of elements of $\Delta_Y$, so that $\Delta_Y$ spans $V$ over $\Q$. Thus there exist 
$-2$-curves $C_1, \dots, C_k$ such that $V$ is spanned by the classes $[C_i]$,   and the  intersection matrix $(C_i\cdot C_j)$ is negative definite.  The classes $[C_1], \dots, [C_k], [D_1], \dots, [D_r]$ span a negative definite sublattice of $H^2(Y; \Zee)$. By Lemma~\ref{constructnef}  there exists a nef and big divisor $H$ such that $H$ is perpendicular to the curves $C_1, \dots, C_k, D_1, \dots, D_r$. Clearly, then, $[H] \in \overline{\mathcal{A}}(Y) \subseteq \overline{\mathcal{A}}_{\text{\rm{gen}}}$  and $[H] = tx$ for some $t\in \Ar^+$. Hence $x\in \overline{\mathcal{A}}_{\text{\rm{gen}}}$ as well. Note that $[H]^\perp$ is spanned over $\Q$ by $[C_1], \dots, [C_k], [D_1], \dots, [D_r]$.

Since $x\in \overline{\mathcal{A}}(Y)$, if $\alpha$ is effective, $x\cdot \alpha \geq 0$. Conversely, suppose that $\alpha$ is a numerical exceptional curve with $x\cdot \alpha \geq 0$ and that $\alpha$ is not   effective. Then $-\alpha + [K_Y]=[G]$, where $G$ is effective, and $H\cdot (-\alpha + [K_Y]) = -\alpha \cdot [H] \leq 0$. Hence $(-\alpha + [K_Y]) \cdot [H] = 0$.  

\begin{claim} $-\alpha + [K_Y] = \sum_ia_i[C_i] + \sum_jb_j[D_j]$ where the $a_i, b_j \in \Zee$.
\end{claim} 
\begin{proof}[Proof of the claim] In any case, since $-\alpha + [K_Y]$ is perpendicular to $[H]$, there exist $a_i, b_j \in \Q$ such that $-\alpha + [K_Y] = \sum_ia_i[C_i] + \sum_jb_j[D_j]$.   Write $-\alpha + [K_Y]= [G] = \sum_in_i[C_i] + \sum_j m_j[D_j] +[F]$, where   $n_i, m_j \in \Zee$ and $F$ is an effective curve not containing $C_i$ or $D_j$ in its support for any $i,j$. By (ii) of Lemma~\ref{negdef}, $F=0$, $a_i = n_i$ and $b_j= m_j$ for all $i,j$.  Hence $a_i, b_j \in \Zee$.
\end{proof}

  Since $-\alpha + [K_Y]$ is an integral linear combination of the $[C_i]$ and $[D_j]$, the same holds for  $\alpha$. Then $\alpha = \sum_ic_i[C_i] + \sum_jd_j[D_j]$ with $c_i, d_j\in \Zee$. But $\alpha ^2 =-1 = (\sum_ic_iC_i)^2 + (\sum_jd_jD_j)^2$. Both terms are non-positive, and so $(\sum_ic_iC_i)^2 \geq -1$. But if $\sum_ic_iC_i \neq 0$, then $(\sum_ic_iC_i)^2 \leq -2$. Thus $\sum_ic_iC_i =0$ and  $\alpha$ lies in the span of the $[D_j]$. Conversely, if $\alpha$ is not in the span of the $[D_j]$ and $\alpha \cdot x \geq 0$, then $\alpha$ is the class of an effective curve.
\end{proof}

\begin{proof}[Proof of Theorem~\ref{disttheorem}] It follows from Proposition~\ref{aprop} that, if $x\in \mathcal{C}^+(Y) \cap \Lambda(Y,D)$ is $R_Y$-distinguished, then $\overline{\mathcal{A}}_{\text{\rm{gen}}}(Y)$ is the set of all $y\in \mathcal{C}^+_D(Y)$ such that $\alpha \cdot y \geq 0$ for all $\alpha$ a numerical exceptional curve on $Y$, not in the span of the $[D_i]$, such that $\alpha \cdot x \geq 0$.  Let $f$ be an isometry satisfying the conditions of the theorem. Then $f(x)$ is $R_{Y'}$-distinguished, and $f(\overline{\mathcal{A}}_{\text{\rm{gen}}}(Y))$ is clearly the set of all $y\in \mathcal{C}^+_{D'}(Y')$ such that $\alpha \cdot y \geq 0$ for all $\alpha$ a numerical exceptional curve on $Y'$, not in the span of the $[D_i']$, such that $\alpha \cdot f(x)  \geq 0$. Again by Proposition~\ref{aprop}, this set is exactly $\overline{\mathcal{A}}_{\text{\rm{gen}}}(Y')$.
\end{proof}

Theorem~\ref{disttheorem} covers all of the cases in \cite{Looij} except for the case of $5$ components: By inspection of the root diagrams on pp.\ 275--277 of \cite{Looij}, the complement of any trivalent vertex spans a negative definite codimension one subspace, except in the case of $5$ components. To give a direct argument along the above lines which also handles this case (and all of the other cases in \cite{Looij}), we recall the basic setup there: There exists a subset $B= \{\beta_1, \dots, \beta_n\}  \subseteq R$ such that $B$ is a basis for $\Lambda\otimes \Ar$, and there exist $n_i\in \Zee^+$ such that $(\sum_in_i\beta_i) \cdot \beta_j > 0$ for all $j$ (compare also \cite{Looijpre} (1.18)). In particular, note that the intersection matrix $(\beta_i\cdot \beta_j)$ is non-singular. Finally, by the  classification of Theorem (1.1) in \cite{Looij}, there exists a deformation of $(Y,D)$ for which $\beta_i = [C_i]$ is the class of a $-2$-curve for all $i$. (With some care, this explicit argument could be avoided by appealing to the surjectivity of the period map and (i) of Theorem~\ref{mainprop}.)

\begin{theorem}    Let $(Y,D)$ and $(Y', D')$ be two anticanonical pairs satisfying the hypotheses of the preceding paragraph, both negative definite,   with $r(D) = r(D')$, and let $f\colon H^2(Y; \Zee) \to H^2(Y'; \Zee)$ be an isometry such that $f([D_i]) = [D_i']$ for all $i$, $f (\mathcal{C}^+(Y)) =  \mathcal{C}^+(Y')$, and $f(R_Y) = R_{Y'}$.  Then $f(\overline{\mathcal{A}}_{\text{\rm{gen}}}(Y)) = \overline{\mathcal{A}}_{\text{\rm{gen}}}(Y')$.
\end{theorem}
\begin{proof} (Sketch) With notation as in the paragraph preceding the statement of the theorem, let $h =\sum_in_i\beta_i$ have the  property that $h\cdot \beta_i > 0$. By the arguments used in the proof of Theorem~\ref{disttheorem}, it is enough to show that $h\in \overline{\mathcal{A}}_{\text{\rm{gen}}}$ and that, if $\alpha$ is a numerical exceptional curve  and $\alpha$ is not in the span of the $[D_j]$, then $\alpha$ is  effective   $\iff$ $\alpha \cdot h \geq 0$. Also, it is enough to prove this for some deformation of $(Y,D)$, so we can assume $\beta_i = [C_i]$ is the class of a $-2$-curve for all $i$, hence that $h$ is the class of $H=\sum_in_iC_i$. By construction, $H\cdot C_j > 0$ for every $j$, hence $H$ is nef and big. By   Lemma~\ref{remarkafternumeric}, it is enough to show that, if $G$ is an irreducible curve not equal to $D_i$ for any $i$, then $H\cdot G > 0$. Since $H$ is nef, it suffices to rule out the case $H\cdot G =0$, in which case $G^2 < 0$. As $G\neq D_j$ for any $j$, then $G$ is either a $-2$-curve or an exceptional curve. The case where $G$ is a $-2$-curve is impossible since then $G$ is orthogonal to the span of the $[C_i]$, but the $[C_i]$ span $\Lambda$ over $\Q$ and the intersection form is nondegenerate. So $G=E$ is an exceptional curve disjoint from the $C_i$. If $(\overline{Y}, \overline{D})$ is the anticanonical pair obtained by contracting $E$, then the $[C_i]$ define classes in  $\overline{\Lambda} = \Lambda(\overline{Y}, \overline{D})$. Since the intersection form $(C_i\cdot C_j)$ is nondegenerate, the rank of $\overline{\Lambda}$ is at least that of the rank of $\Lambda$. It is easy to check that the classes of $\overline{D}_1, \dots, \overline{D}_r$ are linearly independent: if say $E$ meets $D_1$, then the intersection matrix of $\overline{D}_2, \dots, \overline{D}_r$ is still negative definite, and then (ii) of Lemma~\ref{negdef} (with $F = \overline{D}_1$ and $G_1, \dots, G_n = \overline{D}_2, \dots, \overline{D}_r$) shows that the classes of $\overline{D}_1, \dots, \overline{D}_r$ are linearly independent. Hence the rank of $H^2(\overline{Y}; \Zee)$ is greater than or equal to the rank of $H^2(Y; \Zee)$, which contradicts the fact that $\overline{Y}$ is obtained from $Y$ by contracting an exceptional curve.
\end{proof}

\section{Some  examples}

\begin{example} We give a series of examples satisfying the hypotheses of Theorem~\ref{disttheorem} where the number of components and the multiplicities are arbitrarily large. Let $(\overline{Y}, \overline{D})$ be the anticanonical pair obtained by making $k+6$ infinitely near blowups starting with the double point of a nodal cubic. Thus $\overline{D} = \overline{D}_0 + \cdots + \overline{D}_{k+6}$, where $\overline{D}_0^2 = -k$, $\overline{D}_i^2 = -2$, $1\leq i\leq k+5$, and  $\overline{D}_{k+6}^2 = -1$. Now blow up $N \geq 1$ points $p_1, \dots, p_N$ on $\overline{D}_{k+6}$, and let $(Y,D)$ be the resulting anticanonical pair. Note that $(Y,D)$ is negative definite as long as $k\geq 3$ or $k=2$ and $N\geq 2$. Clearly $r(D) = k+7$ and $K_Y^2 = 3-k-N$. It follows that $\Lambda = \Lambda(Y,D)$ has rank $N$. If $E_1, \dots, E_N$ are the exceptional curves corresponding to $p_1, \dots, p_N$, then the classes $[E_i] - [E_{i+1}]$ span a negative definite root lattice of type $A_{N-1}$ in $\Lambda$. By making all of the blowups infinitely near to the first point, we see that all of the classes $[E_i] - [E_{i+1}]$ lie in $R$. Hence $(Y,D)$ satisfies the hypotheses of Theorem~\ref{disttheorem}.
\end{example} 

Next we turn to examples where the rank of $\Lambda$ is small.  The case where the rank of $\Lambda$ is $1$ is covered by Theorem~\ref{disttheorem}, as well as the case where the rank of $\Lambda$ is $2$ and $R\neq \emptyset$. Note that, conjecturally at least, the case where $R\neq \emptyset$ should be related to the question of whether the dual cusp singularity deforms to an ordinary double point. It is easy to construct  examples where the rank of $\Lambda$ is $2$ and with $R\neq \emptyset$: begin with an anticanonical pair $(\hat{Y}, \hat{D})$ where the rank of $\Lambda(\hat{Y}, \hat{D})$ is $1$, locate a component $\hat{D}_i$ such that there exists an exceptional curve $E$ on $\hat{Y}$ with $E \cdot \hat{D}_i = 1$, and blow up a point of $\hat{D}_i$ to obtain a new anticanonical pair $(Y,D)$ together with  exceptional curves $E,E'$ (where we continue to denote by $E$ the pullback to $Y$ and by $E'$ the new exceptional curve), such that $[E] -[E'] \in R$. So our  interest is in finding examples where $R=\emptyset$.

\begin{remark} In case the rank of $\Lambda$ is $2$ and $R\neq \emptyset$, it is easy to see that either $(\overline{\mathcal{A}}_{\text{\rm{gen}}} \cap \Lambda)/\Ar^+$ is a closed (compact) interval or $\overline{\mathcal{A}}_{\text{\rm{gen}}} \cap \Lambda =\mathcal{C}^+\cap \Lambda$ (and in fact both cases arise). In either case, there is at most one wall $W^\beta$ with $\beta \in R$  passing through the interior of $\overline{\mathcal{A}}_{\text{\rm{gen}}} \cap \Lambda$, and hence either $R=\emptyset$ or $R =\{\pm \beta\}$.
\end{remark}

\begin{example}
 We give an example where the rank of $\Lambda$ is $2$ and there are no $\beta \in \Lambda$ such that $\beta^2 = -2$, in particular $R=\emptyset$, hence the condition $f(R) = R$ is automatic for every isometry $f$, and of an isometry $f$ which preserves $\mathcal{C}^+$ and the classes $[D_i]$ but not the generic ample cone. Let $(\overline{Y}, \overline{D})$ be the anticanonical pair obtained by making $9$ infinitely near blowups starting with the double point of a nodal cubic. Thus $\overline{D} = \overline{D}_0 + \cdots + \overline{D}_9$, where $\overline{D}_0 = 3H -2E_1 -\sum_{i=2}^9E_i$, $\overline{D}_i = E_i - E_{i+1}$, $1\leq i \leq 8$, and $\overline{D}_9 = E_9$. Make two more blowups, one at a point $p_{10}$ on $\overline{D}_9$, and one at a point $p_{11}$ on $\overline{D}_4$. This yields an anticanonical pair $(Y,D)$ with $D_0 = 3H -2E_1 -\sum_{i=2}^9E_i$, $D_i = E_i - E_{i+1}$, $i>0$ and $i\neq 4$, and $D_4 = E_4 - E_5 - E_{11}$. Thus 
$$(-d_0, \dots, -d_9) = (3,2,2,2,3,2,2,2,2,2),$$
i.e.\  $D$ is of type $\displaystyle \begin{pmatrix} 3&3\\3&5\end{pmatrix}$, with dual cycle  $\displaystyle \begin{pmatrix} 6&8\\0&0\end{pmatrix}$ in the notation of \cite{FriedmanMiranda}. Set
\begin{align*}
G_1 &= 5H - 2\sum_{i=1}^4E_i - \sum_{i=5}^{10}E_i -E_{11};\\
G_2 &= 10H-5\sum_{i=1}^4E_i-\sum_{i=5}^{10}E_i -4E_{11}.
\end{align*}
It is straightforward to check that $(G_i\cdot D_j) = 0$ for $i=1,2$ and $0\leq j \leq 9$. Hence $G_1, G_2 \in \Lambda$. Also,
$$G_1^2 = 2; \qquad G_2^2 = -22; \qquad G_1\cdot G_2 = 0.$$
The corresponding quadratic form 
$$q(n,m) = (nG_1 + mG_2)^2 = 2n^2 - 22m^2$$
has discriminant $-44=-2^2\cdot 11$. Note that this is consistent with the fact that the discriminant of the dual cycle is
$$\det \begin{pmatrix} -6&2\\2&-8\end{pmatrix} = 44.$$
It is easy to see that $G_1$ and $G_2$ are linearly independent mod $2$ and hence span a primitive lattice, which must therefore equal $\Lambda$.

First we claim that there is no element of $\Lambda$ of square $-2$. This is equivalent to the statement that there is no solution in integers to the equation $n^2 - 11m^2 = -1$, i.e.\ that the fundamental unit in $\Zee[\sqrt{11}]$ has norm $1$. But clearly if there were an integral solution to $n^2 - 11m^2 = -1$, then since $-11\equiv 1 \bmod 4$, we could write $-1$ as a sum of squares mod $4$, which is impossible. In fact, the fundamental unit in $\Zee[\sqrt{11}]$ is $10 + 3\sqrt{11}$. Thus, if $R$ is the set of roots for $(Y,D)$, then $R=\emptyset$. In particular, any isometry $f$ trivially satisfies: $f(R) = R$.

Finally, we claim that there is an isometry $f$ of $H^2(Y; \Zee)$  such that $f([D_i]) = [D_i]$ for all $i$ and $f(\mathcal{C}^+) = \mathcal{C}^+$, but such that $f$ does not preserve the generic ample cone. Note that the unit group $U$ of $\Zee[\sqrt{11}]$ acts as a group of isometries on $\Lambda$, and hence acts as a group of isometries (with $\Q$-coefficients) of the lattice $H^2(Y; \Q) = (\Lambda \otimes \Q) \oplus \bigoplus_i\Q[D_i]$, fixing the classes $[D_i]$. Also, any isometry of $\Lambda$ which is trivial on the discriminant group $\Lambda\spcheck/\Lambda$ extends to an integral isometry of $H^2(Y; \Zee)$ fixing the $[D_i]$. Concretely, the discriminant form $\Lambda\spcheck/\Lambda \cong \Zee/2\Zee \oplus \Zee/22\Zee$.  If $\mu = 10 + 3\sqrt{11}$, then it is easy to check that the automorphism of $\Lambda$ corresponding to $\mu^2 = 199+ 60\sqrt{11}$ acts trivially on $\Lambda\spcheck/\Lambda$ and hence defines  an isometry $f$ of $H^2(Y; \Zee)$ fixing the   $[D_i]$.  Then $f$ acts freely on $(\mathcal{C}^+\cap \Lambda)/\Ar^+$, which is just a copy of $\Ar$ (and $f$ acts on it via translation). But the intersection of the generic ample cone with $\Lambda$ has the nontrivial wall $W^{E_{11}}$, so that the intersection cannot be all of $\mathcal{C}^+\cap \Lambda$. It then follows that $f^{\pm1}$ does not preserve the generic ample cone. Explicitly,  let $(\hat{Y}, \hat{D})$ be the surface obtained by contracting $E_{11}$ and let $\hat{G}_1 = 4G_1-G_2 = 10H - 3\sum_{i=1}^{10}E_i$ be the pullback of the  positive generator of $\Lambda(\hat{Y}, \hat{D})$. Thus $\hat{G}_1$ is nef and big, so that $\hat{G}_1\in \overline{\mathcal{A}}_{\text{\rm{gen}}}$. Clearly $\hat{G}_1\in W^{E_{11}}$. If $A = \displaystyle \begin{pmatrix} a & 11b\\b&a\end{pmatrix}$ is the isometry of $\Lambda$ corresponding to multiplication by the unit $a + b\sqrt{11}$, then $A(G_1) = aG_1 +bG_2$, $A(G_2) = 11bG_1 + aG_2$,  and $A( \hat{G}_1) = (4a-11b)G_1 + (4b-a)G_2$. Thus
$$E_{11} \cdot A( \hat{G}_1) = (4a-11b) + 4(4b-a) = 5b,$$
hence $E_{11} \cdot A( \hat{G}_1) < 0$ if $b< 0$. Taking $f^{-1}$, which corresponds to $199- 60\sqrt{11}$, we see that  $f^{-1}(\hat{G}_1)\notin \overline{\mathcal{A}}_{\text{\rm{gen}}}$.  
\end{example}

\begin{example} In this example,   the rank of $\Lambda$ is $2$ and $R=\emptyset$, but there exist infinitely many $\beta \in \Lambda$ such that $\beta^2 = -2$. The condition $f(R) = R$ is again automatic for every isometry $f$, and reflection about every $\beta \in \Lambda$ with $\beta^2=-2$ is an isometry which  preserves $\mathcal{C}^+$ and the classes $[D_i]$ but not the generic ample cone.

As in the previous example, let $(\overline{Y}, \overline{D})$ be the anticanonical pair obtained by making $9$ infinitely near blowups starting with the double point of a nodal cubic. Thus $\overline{D} = \overline{D}_0 + \cdots + \overline{D}_9$, where $\overline{D}_0 = 3H -2E_1 -\sum_{i=2}^9E_i$, $\overline{D}_i = E_i - E_{i+1}$, $1\leq i \leq 8$, and $\overline{D}_9 = E_9$. Make two more blowups, one at a point $p_{10}$ on $\overline{D}_9$, and one at a point $p_{11}$ on $\overline{D}_0$. This yields an anticanonical pair $(Y,D)$ with $D_0 = 3H -2E_1 -\sum_{i=2}^9E_i-E_{11}$ and $D_i = E_i - E_{i+1}$, $1\leq i\leq 9$. Thus 
$$(-d_0, \dots, -d_9) = (4,2,2,2,2,2,2,2,2,2),$$
i.e.\  $D$ is of type $\displaystyle \begin{pmatrix} 4\\9\end{pmatrix}$, with dual cycle  $\displaystyle \begin{pmatrix} 12\\1\end{pmatrix}$ in the notation of \cite{FriedmanMiranda}. Set
\begin{align*}
G_1 &= 10H - 3\sum_{i=1}^{10}E_i ;\\
G_2 &= 3H- \sum_{i=1}^{10}E_i +E_{11}.
\end{align*}
It is straightforward to check that $(G_i\cdot D_j) = 0$ for $i=1,2$ and $0\leq j \leq 9$. Hence $G_1, G_2 \in \Lambda$. Also,
$$G_1^2 = 10; \qquad G_2^2 = -2 ; \qquad G_1\cdot G_2 = 0.$$
The corresponding quadratic form 
$$q(n,m) = (nG_1 + mG_2)^2 = 10n^2 -  2m^2$$
has discriminant $-20=-2^2\cdot 5$. Note that this is consistent with the fact that the discriminant of the dual cycle is
$$\det \begin{pmatrix} -12&2\\2&-2\end{pmatrix} = 20.$$
It is easy to see that $G_1$ and $G_2$ are linearly independent mod $2$ and hence span a primitive lattice, which must therefore equal $\Lambda$.

To give a partial description of $\overline{\mathcal{A}}_{\text{\rm{gen}}} \cap \Lambda$, note that (as  for $\hat{G}_1$ in the previous example) $G_1$ is the pullback to $Y$ of a positive generator for $\Lambda(\hat{Y}, \hat{D})$, where $\hat{Y}$ denotes the surface obtained by contracting $E_{11}$. Thus $G_1$ is nef and big, so that $G_1\in \overline{\mathcal{A}}_{\text{\rm{gen}}}$ and also $G_1\in W^{E_{11}}$. Hence 
$$\mathcal{C}^+\cap \Lambda = \{nG_1+ mG_2: 5n^2 - m^2 > 0, n>0\},$$
i.e.\ $n>0$ and $-n\sqrt{5} < m < n\sqrt{5}$. 
The condition $E_{11} \cdot (nG_1+ mG_2) \geq 0$ gives $m\leq 0$. To get a second inequality on $n$ and $m$, let
$$E' = 5H - 4E_{11} - \sum_{i=1}^{10}E_i.$$
Then $(E')^2 = E' \cdot K_Y = -1$, and $H\cdot E' > 0$. Hence $E'$ is effective. (In fact one can show that $E'$ is generically the class of an exceptional curve.) Thus, for all $nG_1+ mG_2 \in \overline{\mathcal{A}}_{\text{\rm{gen}}}$,
$$E' \cdot (nG_1+ mG_2) = 20n + 9m \geq 0,$$
hence 
$$\overline{\mathcal{A}}_{\text{\rm{gen}}} \cap \Lambda \subseteq \{nG_1 + mG_2: n > 0, -{\textstyle\frac{20}{9}}n \leq m \leq 0.\}.$$
Next we describe the classes $\beta\in \Lambda$ with $\beta^2 = -2$. The element $\beta = aG_1 + bG_2\in \Lambda$ satisfies $\beta^2 = -2$ $\iff$ $5a^2 -b^2 =-1$, i.e.\ $\iff$ $b+ a\sqrt{5}$ is a unit in the (non-integrally closed) ring $\Zee[\sqrt{5}]$. For example, the class $G_2$ corresponds to $1$; as we have seen, the wall $W^{G_2} = W^{E_{11}}$. The fundamental unit in $\Zee[\sqrt{5}]$ is easily checked to be $9 + 4\sqrt{5}$. However, since we are only concerned with walls which are rays in the fourth quadrant $\{(nG_1+ mG_2): n > 0, m< 0\}$, we shall consider instead $\pm(9-4\sqrt{5})$, and shall choose the sign corresponding to $\beta = 4G_1 - 9G_2$. Note that
$$\beta\cdot (nG_1+ mG_2) = 40n + 18m = 0 \iff E'\cdot (nG_1+ mG_2) = 0.$$
Hence $W^\beta = W^{E'}$. Moreover, for every $\gamma\in \Lambda$ such that $\gamma^2 = -2$ and such that the wall $W^\gamma$ passes through the fourth quadrant, either $W^\gamma =W^\beta$ or  the corresponding ray $W^\gamma$ lies below $W^\beta$. Thus, for every $\gamma \in \Lambda$ with $\gamma^2 = -2$, $r_\gamma$ does not preserve $\overline{\mathcal{A}}_{\text{\rm{gen}}} \cap \Lambda$. Hence $R=\emptyset$. 

Note that, aside from the isometries $r_\beta$, where $\beta^2 = -2$, one can also construct isometries of infinite order preserving $\mathcal{C}^+$ and the classes $[D_i]$ which do not fix preserve $\overline{\mathcal{A}}_{\text{\rm{gen}}}$ using multiplication by  fundamental units in $\Zee[\sqrt{5}]$, as in the previous example.
\end{example}

\begin{remark} The exceptional curve $E'$ used in the above example is part of a general series of such. For $n\geq 0$, let $Y$ be the blowup of $\Pee^2$ at $2n+1$ points $p_0, \dots, p_{2n}$, with corresponding exceptional curves $E_0, \dots, E_{2n}$, and consider the divisor
$$A= nH - (n-1)E_0  - \sum_{i=1}^{2n}E_i.$$
Then $A^2 = A\cdot K_Y = -1$, and it is easy to see that there exist $p_0, \dots, p_{2n}$ such that $A$ is the class of an exceptional curve. In fact, if $\mathbb{F}_1$ is the blowup of $\Pee^2$ at $p_0$, then $\Sigma = nH - (n-1)E_0$ is very ample on $\mathbb{F}_1$ and, for an anticanonical divisor  $D\in |-K_{\mathbb{F}_1}| = |3H-E_0|$,  $\Sigma \cdot D = 2n+1$. From this it is easy to see that we can choose the points $p_1, \dots, p_{2n}$ to lie on the image of $D$ in $\Pee^2$, and hence we can arrange the blowup $Y$ to have (for example) an irreducible anticanonical nodal curve.
\end{remark}

\end{document}